\documentclass[a4paper,11pt]{article}

\usepackage{amsmath,amsthm,amssymb}
\usepackage[margin=1in]{geometry}
\usepackage{mathtools}
\usepackage{mathrsfs}
\usepackage{hyperref}

\newtheorem{proposition}{Proposition}
\newtheorem{corollary}{Corollary}
\newtheorem{theorem}{Theorem}

\newtheorem{remark}{Remark}
\newtheorem{question}{Question}

\newtheorem{conjecture}{Conjecture}

\DeclareMathOperator{\const}{const}
\newcommand{\prob}[1]{\mathsf{#1}}
\begin{document}
\title{Growth of bilinear maps III: Decidability}
\author{Vuong Bui\footnote{LIRMM, Universit\'e de Montpellier, CNRS, 161 Rue Ada, 34095 Montpellier, France and UET, Vietnam National University, Hanoi, 144 Xuan Thuy Street, Hanoi 100000, Vietnam (\texttt{bui.vuong@yandex.ru})}}
\date{}

\maketitle
\begin{abstract}
	The following notion of growth rate can be seen as a generalization of joint spectral radius: Given a bilinear map $*:\mathbb R^d\times\mathbb R^d\to\mathbb R^d$ with nonnegative coefficients and a nonnegative vector $s\in\mathbb R^d$, denote by $g(n)$ the largest possible entry of a vector obtained by combining $n$ instances of $s$ using $n-1$ applications of $*$. Let $\lambda$ denote the growth rate $\limsup_{n\to\infty} \sqrt[n]{g(n)}$. Rosenfeld showed that the problem of checking $\lambda\le 1$ is undecidable by reducing the problem of joint spectral radius. 

	In this article, we provide a simpler reduction using the observation that matrix multiplication is actually a bilinear map. Moreover, we extend the reduction to show that checking $\lambda\le 1$ is still undecidable even if $s$ is positive. If there is no restriction on the signs, we can also show that the problem of checking if the system can produce a zero vector is undecidable by reducing the problem of checking the mortality of a pair of matrices. This answers a question asked by Rosenfeld. Beside that, we confirm a remark of Rosenfeld that the problem does not become harder when we introduce more bilinear maps and more starting vectors. 

	It is known that if the vector $s$ is strictly positive, then the limit superior $\lambda$ is actually a limit. However, we show that when $s$ is only nonnegative, the problem of checking the existence of the limit is undecidable. 
	This also answers a question asked by Rosenfeld.

	We provide a formula for the growth rate $\lambda$ in terms of the diagonals of matrices corresponding to a special structure called ``linear pattern''. A condition is given so that the limit $\lambda$ exists. This actually provides a simpler proof for the existence of the limit $\lambda$ when $s>0$. 
	An important corollary of the formula is the computability of the growth rate, which answers another question by Rosenfeld. 
	Another corollary is that the problem of checking $\lambda\le 1$ is reducible to the problem of joint spectral radius, via the halting problem, i.e., the two problems are Turing equivalent. Also, we relate the finiteness property of a set of matrices to the notion ``linear pattern'' of a bilinear system.

\end{abstract}
\section{Introduction}
Given a bilinear map $*:\mathbb R^d\times\mathbb R^d\to\mathbb R^d$ with \emph{nonnegative} coefficients and a \emph{positive} vector $s\in\mathbb R^d$, denote by $g(n)$ the largest possible entry of a vector obtained by combining $n$ instances of $s$ using $n-1$ applications of $*$. For example, all the combinations of $4$ instances of $s$ are
\[
	s*(s*(s*s)), s*((s*s)*s), (s*s)*(s*s), (s*(s*s))*s, ((s*s)*s)*s.
\]
It was shown in \cite{bui2021growth} that the following growth rate exists:
\[
	\lambda=\lim_{n\to\infty} \sqrt[n]{g(n)}.
\]

When the entries of $s$ are not positive but only \emph{nonnegative}, the limit $\lambda$ may no longer exist. However, relaxing the requirements on the signs in this way is often asked in applications. Therefore, Rosenfeld \cite{rosenfeld2022undecidable} extends the notion of the growth rate $\lambda$ for the case $s$ is nonnegative by defining
\[
	\lambda=\limsup_{n\to\infty} \sqrt[n]{g(n)},
\]
which is called the \emph{growth rate of the bilinear system} $(*,s)$.

Let us call the former setting the positive setting and the latter setting the nonnegative setting (with respect to the sign of $s$).

The study of this problem was first started by Rote \cite{rote2019maximum} with the maximum number of minimal dominating sets in a tree of $n$ leaves as an example. Later on, a richer set of applications to the maximum number of different types of dominating sets, perfect codes, different types of matchings, and maximal irredundant sets in a tree was given by Rosenfeld \cite{rosenfeld2021growth}. A restricted class of the problem was shown to be decidable in \cite{bui2022replacements}, where the growth rate can be computed precisely provided that the entries are rational.

For application purpose, estimating $\lambda$ is a natural problem. In \cite{bui2021growth2} the limit in the positive setting can be approximated to an arbitrary precision. In \cite{rosenfeld2021growth} the growth rate in the nonnegative setting was shown to be upper semi-computable, i.e., we can generate a sequence of upper bounds converging to $\lambda$. In this article, we show that the growth rate in the nonnegative setting is also lower semi-computable, that is the growth rate is computable. However, it still remains the problem of checking if $\lambda\le 1$. In \cite{rosenfeld2022undecidable} Rosenfeld shows that checking $\lambda\le 1$ is undecidable for the nonnegative setting by reducing the problem of checking $\rho\le 1$ for the joint spectral radius $\rho$. The notion of joint spectral radius was first introduced in \cite{rota1960note} and the growth of bilinear maps can be seen as a generalization.

In this paper, we provide another proof with a simpler reduction using the observation that matrix multiplication is also a bilinear map, as in Theorem~\ref{thm:undecidable-estimate}. The reduction is quite natural, and the products of the matrices can be found in an embedded form in the resulting vectors. Moreover, we extend the reduction to show that checking $\lambda\le 1$ is still undecidable in the positive setting. We prove its undecidability by reducing the problem of checking $\rho\le 1$ for the joint spectral radius $\rho$ of a pair of \emph{positive} matrices, as in Section~\ref{sec:undecidable-positive-setting}. The strict positivity has however brought some complications.

The undecidability of the problem for joint spectral radius is actually proved by
\[
	\prob{HP} \le\dots\le \prob{PFAE} \le \prob{JSR},
\]
where $\prob{HP}$ denotes the halting problem, $\prob{PFAE}$ denotes the problem of probabilistic finite automaton emptiness and $\prob{JSR}$ denotes the problem of joint spectral radius. (We denote $A\le B$ if Problem $A$ can be reduced to Problem $B$.) We do not fill the dots between $\prob{HP}$ and $\prob{PFAE}$ since two problems can be filled there: (i) Post's Correspondence Problem, or (ii) Halting Problem for $2$-Counter Machines (see \cite{rote2024probabilistic} for more details).

To be more precise, we mean by $\prob{JSR}$ the problem of checking $\rho\le 1$.
Note that actually all these problems are Turing equivalent (i.e. each is reducible to any other) since we have a reduction from $\prob{JSR}$ to $\prob{HP}$ by the joint spectral radius theorem (conjectured in \cite{daubechies1992sets}, first proved in \cite{berger1992bounded}), which states that for a finite set $\Sigma$ of matrices we have
\begin{equation}\label{eq:jsr-theorem}
	\rho(\Sigma)=\sup_n\max_{A_1,\dots,A_n\in\Sigma}\sqrt[n]{\rho(A_1\dots A_n)},
\end{equation}
where $\rho$ denotes both the joint spectral radius of a set of matrices and the ordinary spectral radius of a matrix, depending on the argument. Indeed, we just run the program that looks for a sequence of matrices whose product has the spectral radius greater than $1$. The program does not stop if and only if $\rho(\Sigma)\le 1$. One may say that the problem is semidecidable. (Note that the problem for spectral radius ($\prob{SR}$) is decidable by Tarski's decision procedure for the
first-order theory of the reals.)

On the other hand, a formula of $\lambda$ in Section~\ref{sec:nonnegative-setting} allows a reduction from checking $\lambda\le 1$ to the halting problem. We show the formula in the following form, which looks similar to the joint spectral radius theorem:
\begin{equation}\label{eq:formula-for-lambda}
	\lambda=\sup_n \max_{\substack{\text{linear pattern $P$}\\ |P|=n}} \sqrt[n]{\rho(M(P))}.
\end{equation}
We do not explain the terms in detail, which is done in Section~\ref{sec:nonnegative-setting}, but we may say roughly that a linear pattern is a sequence $x_n$ for $n=0,1,2,\dots$ so that $x_0=s$ and $x_n$ for $n\ge 1$ is a combination of some instances of $s$ and only \emph{one} instance of $x_{n-1}$. The notation $|P|$ denotes the number of instances of $s$ and the matrix $M=M(P)$ represents the linear relation $x_n=Mx_{n-1}$ for every $n\ge 1$. Note that there is a finite number of linear patterns of given size $|P|=n$ and $M(P)$ is a matrix that can be computed for $P$. The reduction from checking $\lambda\le 1$ to the halting problem is done similiar to the one for the problem of checking $\rho(\Sigma)\le 1$.

We have established the relation of these problems to the problem of the growth rate of a bilinear system ($\prob{GRBS}$). An interesting point is that using reductions of the same kind as the one for $\prob{JSR}\le \prob{GRBS}$ we can show that the problem for the growth rate does not become harder when multiple operators and multiple starting vectors are allowed. This was first remarked by Rosenfeld in \cite{rosenfeld2022undecidable}. Let us call it the joint growth rate of a bilinear system ($\prob{JGRBS}$). In total, we have
\[
	\prob{SR} < \prob{HP} = \prob{PFAE} = \prob{JSR} = \prob{GRBS} = \prob{JGRBS},
\]
where $A<B$ means $A\le B$ but we do not have $B\le A$, and $A=B$ means $A\le B$ and $B\le A$, that is each of $A,B$ is reducible to the other, i.e., they are Turing equivalent. 

Note that we still do not yet have a natural reduction from $\prob{GRBS}$ to $\prob{JSR}$ as the one for $\prob{JSR}\le \prob{GRBS}$. Such a reduction is very desirable with some consequences, as discussed in Section~\ref{sec:applications}.

In \cite{rosenfeld2022undecidable} Rosenfeld asks the following question: Suppose the coefficients of $*$ and the entries of $s$ have no condition on the signs (they can even be complex), then is the problem of checking if the system can produce a zero vector decidable? A negative answer is given in Theorem~\ref{thm:undecidable-mortal}. The problem to be reduced is checking the mortality of a pair of matrices. 

Since the reduction for $\prob{JSR}\le \prob{GRBS}$ is quite natural, we can relate the finiteness property \cite{lagarias1995finiteness} for the joint spectral radius to a result on whether the rate of a linear pattern can attain the growth rate. 
A set $\Sigma$ of matrices is said to have the finiteness property if the supremum in \eqref{eq:jsr-theorem} for the joint spectral radius theorem is attainable, that is there exist $A_1,\dots,A_n\in\Sigma$ so that $\sqrt[n]{\rho(A_1\dots A_n)}=\rho(\Sigma)$. Meanwhile, the rate of a linear pattern $P$ is $\sqrt[|P|]{\rho(M(P))}$ and the supremum in \eqref{eq:formula-for-lambda} is not always attainable, that is there exists a system where no linear pattern has the same rate as the growth rate (e.g. see \cite{bui2021growth}). 
The relation of the finiteness property and whether a linear pattern attains the growth rate is presented in Section~\ref{sec:applications}.

Checking if the limit superior $\lambda$ is actually a limit is also an interesting problem, whose decidability was asked by Rosenfeld in a correspondence. Theorem~\ref{thm:undecidable-validity} shows that it is undecidable by reducing the problem of checking $\lambda\le 1$.
During the course, there is a transform of $(*,s)$ to a new system with the corresponding function $g'(n)$ so that for every $m\ge 1$ we have $g'(2m)=g(m)$ and $g'(2m+1)=0$.

As an attempt to study the nonnegative setting, we extend the formula of $\lambda$ in \cite{bui2021growth2} to the nonnegative setting in Section~\ref{sec:nonnegative-setting}. Using the formula, we give a condition so that the limit always exists. This actually serves as a proof of the existence of the limit $\lambda$ in the positive setting, which is simpler than the proof in \cite{bui2021growth}. Another corollary is a transform so that the new system still has the same growth rate as the original one but with the existence of the limit. In fact, the computability of the growth rate in the nonnegative setting is derived from the formula, as in Theorem~\ref{thm:computable}.

\section{Checking $\lambda\le 1$ is undecidable}
\label{sec:reduction}
We provide in this section a simpler proof of the following result.
\begin{theorem}[Rosenfeld \cite{rosenfeld2022undecidable}]
\label{thm:undecidable-estimate}
	The problem of checking if $\lambda\le 1$ for the nonnegative setting is undecidable.
\end{theorem}
The proof in \cite{rosenfeld2022undecidable} reduces the problem of joint spectral radius to this problem. Our proof also reduces from joint spectral radius but in a much simpler way, by observing that matrix multiplication is also a bilinear map.

The reduction in this section is quite important in the sense that its variants appear throughout the article.
Before going to our proof, we remind the readers the notion of the joint spectral radius.

Given a set of matrices $\Sigma$ in $\mathbb R^d$, the joint spectral radius $\rho(\Sigma)$ of $\Sigma$ is defined to be the limit
\[
	\rho(\Sigma)=\lim_{n\to\infty} \sqrt[n]{\max_{M_1,\dots,M_n\in\Sigma} \|M_1\dots M_n\|}.
\]
The limit was introduced with a proof of its existence in \cite{rota1960note} and it is independent of the norm. For convenience, we let the norm be the maximum norm, i.e., the largest absolute value of an entry in the matrix.

\begin{proof}[Simpler proof of Theorem \ref{thm:undecidable-estimate}]
Consider the problem of checking if $\rho(\{A,B\})\le 1$ for the joint spectral radius $\rho(\{A,B\})$ of a pair of nonnegative matrices $A,B$ in $\mathbb R^{d\times d}$, which is known to be undecidable \cite[Theorem~$2$]{blondel2000boundedness}. We reduce this problem to the problem of checking if $\lambda\le 1$ for the bilinear system $(*,s)$ constructed as follows.

We use some embedding of a $d\times d$ matrix $C$ to a vector $v$ in the space $\mathbb R^{d^2}$, and allow ourselves to write $(C,i,j)$ to present a vector in $\mathbb R^{d^2+2}$, where $C$ is embedded in the first $d^2$ dimensions and $i,j$ are the two last dimensions.

Given a pair of matrices $A,B$ in $\mathbb R^d$, we consider the system $(*,s)$ with the $(d^2+2)$-dimensional vector $s=(\mathbf O,1,0)$ for the zero matrix $\mathbf O$ and $*:\mathbb R^{d^2+2}\times\mathbb R^{d^2+2}\to\mathbb R^{d^2+2}$ presented by
\begin{equation}\label{eq:simplified-version}
    \begin{pmatrix}
        C\\i\\j
    \end{pmatrix}
    *
    \begin{pmatrix}
        C'\\i'\\j'
    \end{pmatrix}
    =
    \begin{pmatrix}
        CC' + ij'A + ji'B\\
        0\\
        ii'
    \end{pmatrix}
    ,
\end{equation}
where $CC'$ is the usual matrix multiplication. The key point here is that a matrix multiplication in $\mathbb R^d$ is also a bilinear map in $\mathbb R^{d^2}\times\mathbb R^{d^2}\to\mathbb R^{d^2}$.

Let us write down some beginning combinations:
\begin{align} \label{eq:some-beginning-combinations}
    \begin{split}
    s&=(\mathbf O,1,0)\\
    s*s&=(\mathbf O,0,1)\\
    s*(s*s)&= (\mathbf O\mathbf O + 1\cdot 1\cdot A + 0\cdot 0\cdot B, 0, 0) = (A,0,0)\\
    (s*s)*s&= (\mathbf O\mathbf O + 0\cdot 0\cdot A + 1\cdot 1\cdot B, 0, 0) = (B,0,0)\\
    (s*s)*(s*s)&=(\mathbf O\mathbf O + 0\cdot 1 \cdot A + 1\cdot 0\cdot B, 0, 0) = (\mathbf O,0,0).
    \end{split}
\end{align}

Let $n$ be the number of instances of $s$ in a combination with the resulting vector $v$.
Obviously, $v_{d^2+1}$ (the index of $i$) is nonzero for only $n=1$, and $v_{d^2+2}$ (the index of $j$) is nonzero for only $n=2$. 
It follows that the sum $ij'A + ji'B$ in \eqref{eq:simplified-version} is nonzero only for $n=3$. In other words, whenever $n\ge 4$, the expression for the first $d^2$ dimensions in \eqref{eq:simplified-version} has the recursive form $CC'$. Together with \eqref{eq:some-beginning-combinations}, we have the matrix form $M$ of the first $d^2$ dimensions is the product of matrices from $\{\mathbf O,A,B\}$. If $M$ is not zero, then $M$ is the product of matrices from $\{A,B\}$ where the numbers of instances $m_A,m_B$ of $A,B$ respectively correspond to the number of the occurrences of $s*(s*s)$ and $(s*s)*s$, and $3m_A+3m_B=n$. Note that the last $2$ dimensions of these combinations are always zero, due to $n\ge 3$.

    On the other hand, for any sequence of matrices $M_1,\dots,M_t\in\{A,B\}$, the combination
    \[
        (S_1*(S_2*(S_3*(\dots*(S_{t-1}*S_t)\dots),
    \]
    where $S_k=(s*(s*s))$ if $M_k=A$ and $S_k=((s*s)*s)$ if $M_k=B$, for $k=1,\dots,t$, gives a vector whose first $d^2$ dimensions embed the matrix $M_1\dots M_t$, and the last two dimensions are zero.

	It follows from the two above directions that
	\[
		g(3t)=\max_{M_1,\dots,M_t\in\{A,B\}} \|M_1\dots M_t\|,
	\]
 where $\|\cdot\|$ denotes the maximum norm.

	Also, for $n\ge 3$ and $n$ not divisible by $3$, we have
	\[
		g(n)=0.
	\]

	Therefore,
	\[
		\lambda=\sqrt[3]{\rho(\{A,B\})}.
	\]

	We have reduced the problem of the joint spectral radius to the problem of the growth rate. The conclusion on the undecidability follows.
\end{proof}

The variant of checking $\lambda=1$ is also undecidable due to the undecidability of the corresponding problem of checking $\rho= 1$ for the joint spectral radius. In fact, we can reduce the problem of checking $\lambda\le 1$ to the problem of checking $\lambda=1$ by adding an extra dimension that is always $1$. However, the question for $\rho\ge 1$ still remains open (see \cite[Section 2.2.3]{jungers2009joint} for a discussion):
\begin{conjecture}[Blondel and Tsitsiklis 2000 \cite{blondel2000boundedness}]\label{con:rho>=1}
	It is undecidable to check if $\rho\ge 1$ for the joint spectral radius~$\rho$.
\end{conjecture}
The conjecture has applications in the stability of dynamical systems. If it holds, then the problem of checking $\lambda\ge 1$ is also undecidable.

\section{Checking the mortality is undecidable}
\label{sec:mortality}
Problems of other properties of a pair of matrices can be also reduced to the corresponding ones of a bilinear system. The following theorem is one example.

\begin{theorem} \label{thm:undecidable-mortal}
	When there is no condition on the signs of the coefficients and the starting entries, the problem of checking if the system can produce a zero vector is undecidable.
\end{theorem}
\begin{proof}
	We reduce to this problem the problem of checking if a pair of matrices $A,B$ is mortal, that is checking if there exists a sequence of matrices $M_1,\dots,M_k$ drawn from $\{A,B\}$ for some $k$ so that $M_1\dots M_k$ is a zero matrix. The problem of mortality for a pair of matrices is known to be undecidable \cite{blondel1997pair}. 

    For the space $\mathbb R^d$ of $A,B$, we consider the space $\mathbb R^{d^2+2}$ with an embedding of $d\times d$ matrices into the first $d^2$ dimensions. One may write $(C,i,j)$ where $C$ is a matrix to present a vector in $\mathbb R^{d^2+2}$.

    Consider the system $(*,s)$ with the starting vector $s=(\mathbf I,1,0)$ where $\mathbf I$ is the identity  matrix, and $*$ defined by
    \begin{equation*}
        \begin{pmatrix}
        C\\
        i\\
        j   
        \end{pmatrix}
        *
        \begin{pmatrix}
            C'\\
            i'\\
            j'
        \end{pmatrix}
        =
        \begin{pmatrix}
            CC' + ij'(A-\mathbf I) + ji'(B-\mathbf I)\\
            0\\
            ii'
        \end{pmatrix}
        .
    \end{equation*}

    Some begining combinations are
    \begin{align}\label{eq:beginning-mortality}
    \begin{split}
        s&=(\mathbf I,1,0)\\
        s*s &= (\mathbf I,0,1)\\
        s*(s*s) &= (\mathbf I+(A-\mathbf I),0,0) = (A,0,0)\\
        (s*s)*s &= (\mathbf I+(B-\mathbf I),0,0) = (B,0,0)\\
        (s*s)*(s*s) &= (\mathbf I, 0,0).
    \end{split}
    \end{align}

    Consider a vector $v$ obtained by combining $n$ instances of $s$. It follows from $v_{d^2+1} = 0$ for $n > 1$ that $v_{d^2+2} = 0$ for $n>2$. The consequence is that for $n > 3$, the first $d^2$ dimensions of $v$, denoted by $\bar v$, are
    \[
        CC' + ij'(A -\mathbf I) + ji'(B-\mathbf I) = CC'.
    \]
    Together with \eqref{eq:beginning-mortality}, the matrix form of $\bar v$ for any $n$ presents a product of matrices from $\{\mathbf I,A,B\}$. It follows that if $v=0$ for some combination, then $\{A,B\}$ is mortal.
    
    On the other hand, if $\{A,B\}$ is mortal with $M_1\dots M_k=0$, the combination 
    \[
        (S_1*(S_2*(S_3*(\dots*(S_{t-1}*S_t)\dots),
    \]
    where $S_t=(s*(s*s))$ if $M_t=A$, and $S_t=((s*s)*s)$ if $M_t=B$, for $t=1,\dots, k$, is zero.

    The equivalence means that we can reduce the problem of checking the mortality of a pair of matrices to the problem of checking if a bilinear system can produce a zero vector. The conclusion follows.
\end{proof}

\section{Checking if the limit exists is undecidable}
Before showing that it is undecidable to check if the growth rate is a limit, we give the following nice transformation.
\begin{proposition}\label{prop:insert-0-at-odd}
	For every bilinear system $(*,s)$ with the function $g(n)$ we can construct $(*',s')$ so that for every $m\ge 1$ we have $g'(2m+1)=0$ and $g'(2m)=g(m)$, where $g'(n)$ is the function for $(*',s')$.
\end{proposition}
\begin{proof}
	Let $\mathbb R^d$ be the space of $(*,s)$. We write $(x,i)$ for a vector $x\in\mathbb R^d$ and a number $i\in \mathbb R$ to present a vector in $\mathbb R^{d+1}$.
     Consider $(*',s')$ with the $(d+1)$-dimensional vector $s'=(\mathbf{0},1)$ where $\mathbf{0}$ denotes the zero vector and $*':\mathbb R^{d+1}\times\mathbb R^{d+1}\to\mathbb R^{d+1}$ presented by
     \begin{equation}\label{eq:transform-double}
         \begin{pmatrix}
             x\\
             i
         \end{pmatrix}
         *'
         \begin{pmatrix}
             y\\
             j
         \end{pmatrix}
         =
         \begin{pmatrix}
             ijs + x*y\\
             0
         \end{pmatrix}
         .
     \end{equation}
     
    Let $v$ be the vector obtained from a combination of $n$ instances of $s'$ (using $*'$). 
    For $n=1$, we have $v=s'=(\mathbf{0}, 1)$. For $n=2$, we have $v=s'*s'=(s,0)$.
     When $n\ge 3$, the summand $ijs$ in \eqref{eq:transform-double} is zero since either $i$ or $j$ is zero, for which we have the recursive form $x*y$ for the first $d$ dimensions. It follows that the first $d$ dimensions $\bar v$ of $v$ are a combination of vectors in $\{\mathbf{0}, s\}$ (using $*$). If $\bar v$ is nonzero, then $\bar v$ is a combination of some $m$ instances of $s$, with $2m=n$. Since $v_{d+1}=0$ for any $n\ge 2$, we have $g'(2m)\le g(m)$. Considering odd $n>1$, we start with $s*(s*s)=(s*s)*s=(\mathbf 0,0)$ for $n=3$. By induction, one can show that $g'(2m+1)=0$ for any $m\ge 1$. On the other hand, for any combination of $m$ instances of $s$ (using $*$) that is associated with $g(m)$, we also have the corresponding combination of $2m$ instances of $s'$ (using $*'$) by replacing each instance of $s$ by $(s'*s')$. The resulting vector of the former combination is the same as $\bar v$ for the resulting vector $v$ of the latter combination. It follows that $g'(2m)=g(m)$.
     \end{proof}

\begin{theorem}\label{thm:undecidable-validity}
	Checking the existence of the limit of $\sqrt[n]{g(n)}$ is undecidable.
\end{theorem}
\begin{proof}
	We will reduce the problem of checking if $\lambda=\limsup_{n\to\infty} \sqrt[n]{g(n)} \le 1$ for a system $(*,s)$ to the problem of checking the existence of the limit of another system.

	By Proposition~\ref{prop:insert-0-at-odd}, we can construct a system $(*',s')$ so that for every $m\ge 1$ we have $g'(2m+1)=0$ and $g'(2m)=g(m)$. Let the space of $(*',s')$ be $\mathbb R^{d'}$, we construct $*'':\mathbb R^{d'+1}\times\mathbb R^{d'+1}\to\mathbb R^{d'+1}$ and $s''\in\mathbb R^{d'+1}$ so that the system $(*',s')$ is brought into the first $d'$ dimensions of the new system $(*'',s'')$ and 
	\[
		s''_{d'+1}=1,\qquad (x*''y)_{d'+1}=x_{d'+1} y_{d'+1}.
	\]

	The last dimension is obviously always $1$. It follows that $g''(2m+1)=1$ and $g''(2m)=\max\{g(m),1\}$ for $m\ge 1$. 
 It means $\liminf_{n\to\infty} \sqrt[n]{g''(n)}=1$ since $g''(n)\ge 1$ for every $n$ and $\liminf_{n\to\infty} \sqrt[n]{g''(n)} \le \liminf_{m\to\infty} \sqrt[2m+1]{g''(2m+1)} = 1$.
Meanwhile, 
\[
\lambda''=\limsup_{n\to\infty} \sqrt[n]{g''(n)}=\max\left\{\limsup_{n\to\infty} \sqrt[n]{g'(n)}, \limsup_{n\to\infty} \sqrt[n]{g''_{d'+1}(n)}\right\} = \max\{\lambda,1\},
\]
where $g'(n)$ is also the largest possible entry over all the first $d'$ dimensions, and $g''_{d'+1}$ denotes the one for the last dimension, which is always $1$.
In total, the limit of $\sqrt[n]{g''(n)}$ exists if and only if $\lambda\le 1$. 
The reduction is finished, and the conclusion on the undecidability follows.
\end{proof}

\section{Growth rate in the nonnegative setting}
\label{sec:nonnegative-setting}
Before presenting a formula of the growth rate, we present some definitions that can be found in \cite{bui2021growth} and \cite{bui2021growth2}. The definitions here are self-contained, but the readers are advised to check the original source for more intuitions and explanations. In fact, the proof of the formula is a simplified and adapted version of the argument in \cite{bui2021growth2} for the nonnegative setting.

Beside $g(n)$, we also denote by $g_i(n)$ the largest possible $i$-th entry over all vectors obtained from a combination of $n$ instances of $s$.

We make an assumption that for every $i$ there exists some $n$ so that $g_i(n)>0$, otherwise we can safely eliminate such a degenerate dimension $i$. How to check for some $i$ if $g_i(n)=0$ for every $n$ is left as an exercise for the readers. Note that without the assumption, some later results may not hold in their current form.

A \emph{composition tree} is a rooted binary tree where each vertex is assigned a vector in the following way. We assign the same vector $s$ to all leaves, and assign to each non-leaf vertex the value $x*y$ where $x,y$ are respectively the vectors of the left and right children. The vector obtained at the root is called the \emph{vector associated with} the composition tree. We often call a composition tree a \emph{tree} for short. It can be seen that there is a one-to-one correspondence between a tree of $n$ leaves and a combination of $n$ instances of $s$.

If every leaf is assigned the same vector $s$ but a specially \emph{marked leaf} is assigned a vector variable $u$, then the vector $v$ at the root depends linearly on $u$ by a matrix $M=M(P)$, that is $v=Mu$. If the tree is $T$ and the leaf is $\ell$, we say such a setting is a \emph{linear pattern} $P=(T,\ell)$. We call $M$ the \emph{matrix associated with} $P$. 

A \emph{composition} $P_1\oplus P_2$ of two linear patterns $P_1=(T_1,\ell_1),P_2=(T_2,\ell_2)$ is the pattern $(T,\ell)$ so that $T$ is obtained from $T_1$ by replacing $\ell_1$ by $T_2$, and setting $\ell=\ell_2$. If $M_1,M_2$ are the matrices associated with $P_1,P_2$, then $M_1M_2$ is the matrix associated with $P_1\oplus P_2$.

For $m\ge 1$ we denote $P^m=P\oplus\dots\oplus P$ where there are $m$ instances of $P$. If $M$ is associated with $P$ then obviously $M^m$ is associated with $P^m$. 

The \emph{number of leaves} $|P|$ of a pattern $P=(T,\ell)$ is defined to be the number of leaves excluding the marked leaf (i.e. one less than the number of leaves in $T$). One can see that $|P\oplus Q|=|P|+|Q|$.

For convenience, we also denote by $P\oplus T'$ the tree obtained from the tree of the pattern $P$ by replacing the marked leaf by the tree $T'$. Let $u$ be the vector associated with $T'$, the vector $v$ associated with $P\oplus T'$ is $Mu$ for $M=M(P)$. Let $T'$ be a tree with a bounded number of leaves so that $u_j>0$, then 
\begin{equation}\label{eq:matrix_entry_le_g}
    M_{i,j}\le\const M_{i,j}u_j\le\const v_i\le\const g_i(|P|+O(1)).
\end{equation}
Note that $O(1)$ is the number of leaves of $T'$, which could be as large as an exponential function of the dimension, but still bounded anyway.

Let $*$ be represented by the coefficients $c_{i,j}^{(k)}$ so that for any vectors $x,y$ and an index $k$,
\[
	(x*y)_k = \sum_{i,j} c_{i,j}^{(k)} x_i y_j.
\]
The \emph{dependency graph} is the directed graph where the vertices are the dimensions and there is an edge from $k$ to $i$ if and only if $c_{i,j}^{(k)}\ne 0$ or $c_{j,i}^{(k)}\ne 0$ for some $j$ (loops are allowed). The dependency graph can be partitioned into strongly connected components, which we call \emph{components} for short. These components define a partial order so that for two different components $C', C$, we say $C'<C$ if there is a path from $i$ to $j$ for $i\in C$ and $j\in C'$. 

If there is a path from $i$ to $j$, then there is a linear pattern $P_{i\to j}$ of a bounded number of leaves so that $M(P_{i\to j})_{i,j}>0$. It can be seen from the fact that if there is an edge $ki$ then $M(P)_{k,i}>0$ for $P=(T,\ell)$ where $\ell$ is the left (resp. right) child of the root if $c_{i,j}^{(k)}\ne 0$ (resp. $c_{j,i}^{(k)}\ne 0$), and the right (resp. left) subtree has a bounded number of leaves whose associated vector has a positive $j$-th entry. If the distance from $i$ to $j$ is greater than $1$, then the desired linear pattern can be obtained by compositions.

\subsection{The formula}
Now we have enough material to prove the following formula of the growth rate.
\begin{theorem}\label{thm:limsup=sup}
	The growth rate can be expressed as a supremum by
	\[
		\lambda=\limsup_{n\to\infty} \sqrt[n]{g(n)} = \sup_{\text{linear pattern $P$}} \max_i \sqrt[|P|]{M(P)_{i,i}}.
	\]
\end{theorem}
\begin{proof}
	Let $\theta$ denote the supremum in the theorem. It can be seen that $\lambda\ge\theta$. Indeed, for any $P$ and $i$, consider the sequence $n=q|P|+r$ for $q=1,2,\dots$, where $r$ satisfies $g_i(r)>0$ by a tree $T_0$. For such $n$, consider the tree $P^q\oplus T_0$, the associated vector has the $i$-th entry at least $\const (M(P)_{i,i})^q$. As $r$ is bounded, the lower bound of $\lambda$ follows.

	It remains to prove the other direction $\lambda\le\theta$ by the fact that for every $i$ there exists some $r$ so that\footnote{One can follow the induction in the proof to see that $r$ can be bounded by the dimension. In fact, in \cite{bui2021growth2} the bound is even shown to be $\const n^r\theta^n$, but we keep the approach simpler for the purpose of proving the formula only.}
	\begin{equation}\label{eq:upper-bound}
		g_i(n)\le \const n^{O((\log n)^r)} \theta^n.
	\end{equation}

	At first, we make an observation: If $i,j$ are in the same connected component, then for every linear pattern $P$,
	\begin{equation}\label{eq:base}
		M(P)_{i,j}\le\const\theta^{|P|}.
	\end{equation}

	Indeed, let $P_{j\to i}$ be the pattern of a bounded number of leaves so that $M(P_{j\to i})_{j,i}>0$, we have $M(P\oplus P_{j\to i})_{i,i} \ge M(P)_{i,j} M(P_{j\to i})_{j,i} \ge \const M(P)_{i,j}$. Meanwhile, $M(P\oplus P_{j\to i})_{i,i}\le \theta^{|P\oplus P_{j\to i}|} \le\const \theta^{|P|}$. The observation is clarified.

	When the component is not connected (containing a single vertex without loops), the observation is trivial.

	We prove \eqref{eq:upper-bound} by induction on the components. The observation in \eqref{eq:base} is the base case. Indeed, for any $i$ in a minimal component let $P$ be any pattern with the tree associated with $g_i(n)$. We have $g_i(n)=\sum_j M(P)_{i,j} s_j \le\const M(P)_{i,j}$ for some $j$ (note that $j$ is in the same component). Now, suppose \eqref{eq:upper-bound} holds for any vertex in a component lower than the component of $i$ with the degree $r'$, we prove that it also holds for $i$ with some degree $r$.

	Let $T$ be the tree associated with $g_i(n)$. Pick a subtree $T_0$ of $m$ leaves so that $n/3\le m\le 2n/3$. Let the pattern $P'$ be so that we have the decomposition $T=P'\oplus T_0$. Let $M'$ be the matrix associated with $P'$ and $u$ be the vector associated with $T_0$, we have
	\[
		g_i(n)=\sum_j M'_{i,j} u_j \le \const M'_{i,j} u_j \le \const M'_{i,j} g_j(m)
	\]
	for some $j$.

	If $j$ is in the same component as $i$, then we have $M'_{i,j}\le\const\theta^{n-m}$ by \eqref{eq:base}. Therefore,
	\begin{equation}\label{eq:first-case}
    		g_i(n)\le\const \theta^{n-m} g_j(m).
	\end{equation}

	If $j$ is not in the component of $i$, then $g_j(m)\le \const m^{O((\log m)^{r'})} \theta^m$ by induction hypothesis. Since $M'_{i,j}\le \const g_i(|P'|+O(1))$ by \eqref{eq:matrix_entry_le_g}, we have
	\begin{equation}\label{eq:second-case}
    		g_i(n)\le \const g_i(n-m+O(1)) m^{O((\log m)^{r'})} \theta^m. 
	\end{equation}

	In either case we have reduced $n$ to at most a fraction of $n$ and $g_i$ to $g_k$ with $k$ still in the same component of $i$.
    Repeating the process recursively to either $g_j(m)$ or $g_i(n-m+O(1))$ an $O(\log n)$ number of times until the argument is small enough, we obtain
	\[
		g_i(n)\le \const K^{O(\log n)} \left(n^{O((\log n)^{r'})}\right)^{O(\log n)} \theta^{n+O(\log n)} \le \const n^{O((\log n)^r)} \theta^n
	\]
	where $r=r'+1$ and $K$ is the larger one of the constants in \eqref{eq:first-case} and \eqref{eq:second-case}. (Note that $a^{\log b}=b^{\log a}$.)

	The proof finishes by induction. 
\end{proof}

\subsection{A condition for the limit to exist}
We provide a condition in the nonnegative setting so that $\lambda$ is a limit.
\begin{theorem}
	Suppose there exists some $n_0$ so that for every $n\ge n_0$ and for every $i$ we have $g_i(n)>0$, then $\lambda$ is a limit.
\end{theorem}
\begin{proof}
	Let $\theta$ denote the supremum in Theorem~\ref{thm:limsup=sup}, it suffices to prove that 
	\[
		\liminf_{n\to\infty} \sqrt[n]{g(n)}\ge \theta,
	\]
	which can be reduced to showing that for any pattern $P$ and any index $i$, we have
	\[
		\liminf_{n\to\infty} \sqrt[n]{g(n)}\ge \sqrt[|P|]{M(P)_{i,i}}.
	\]

	Indeed, for every $n$ large enough, let $n=q|P|+r$ so that $n_0\le r< n_0+|P|$. 
    Let $T_0$ be the tree associated with $g_i(r)$. The associated vector with $P^q\oplus T_0$ has the $i$-th entry at least a constant times $(M(P)_{i,i})^q$. Since $r$ is bounded, the conclusion follows.
\end{proof}

We have provided another proof that the limit always exists in the positive setting. This proof is simpler than both other versions in \cite{bui2021growth} and \cite{bui2021growth2}.
\begin{corollary}\label{cor:lim-exists}
    The limit $\lambda$ exists in the positive setting, that is we can replace $\limsup$ by $\lim$ in the definition.
\end{corollary}

\section{Applications of the formula}
\label{sec:applications}
\subsection{A formula for the spectral radius}
The formula in Theorem~\ref{thm:limsup=sup} turns out to give more elementary results such as the following formula for the spectral radius of a nonnegative matrix $A$, which is also discussed in detail in \cite{bui2024bound}.
\begin{theorem}\label{thm:spectral-radius}
	For every nonnegative matrix $A$, the spectral radius $\rho(A)$ can be written as
	\[
		\rho(A)=\sup_n\max_i\sqrt[n]{(A^n)_{i,i}}.
	\]
\end{theorem}
\begin{proof}
    The direction that $\rho(A)\ge\sup_n\max_i\sqrt[n]{(A^n)_{i,i}}$ is trivial. We prove the other direction.
    
	Suppose $A$ is a $d\times d$ matrix. Consider an embedding of any $d\times d$ matrix $B$ to a vector $v$ in $\mathbb R^{d^2}$ by the function $\Gamma$ so that
	\[
		B=\Gamma(v),\qquad v=\Gamma^{-1}(B).
	\]
	Let the system $(*,s)$ in the space $\mathbb R^{d^2}$ be so that $s=\Gamma^{-1}(A)$ and
	\[
		u*v=\Gamma^{-1}(\Gamma(u) \Gamma(v)).
	\]
    One can see that every combination of $n$ instances of $s$ gives $\Gamma^{-1}(A^n)$. (The operator $*$ is associative.) Therefore,
	\[
		\lambda=\rho(A).
	\]
	On the other hand, if $P$ is a linear pattern with $|P|=m$, then the relation between the vector at the root $v$ and the vector at the marked leaf $u$ is 
	\[
		\Gamma(v)= A^t \Gamma(u) A^{m-t}
	\]
    for some $0\le t\le m$.
    In particular, for every $i,j$, one can write
    \[
        \Gamma(v)_{i,j} = \sum_{k,\ell} (A^t)_{i,k} \Gamma(u)_{k,\ell} (A^{m-t})_{\ell,j} = \sum_{k,\ell} \Gamma(u)_{k,\ell} (A^t)_{i,k} (A^{m-t})_{\ell,j}.
    \]
 Let $M$ be the $d^2\times d^2$ matrix so that $v=Mu$. The diagonal $M_{(i,j),(i,j)}$ is
 \[
    (A^t)_{i,i} (A^{m-t})_{j,j}.
 \]
	It follows from Theorem~\ref{thm:limsup=sup} that
 \begin{align*}
  \rho(A)=\lambda &=\sup_m \max_{\substack{\text{linear pattern $P$}\\|P|=m}} \max_{i,j} \sqrt[m]{M(P)_{(i,j),(i,j)}} \\
  &\le \sup_m \max_{0\le t\le m} \max_{i,j} \sqrt[m]{(A^t)_{i,i} (A^{m-t})_{j,j}} \\
      &\le \sup_m \max_{0\le t\le m} \max_{i,j} \max\left\{\sqrt[t]{(A^t)_{i,i}}, \sqrt[m-t]{(A^{m-t})_{j,j}}\right\} \\
      &\le \sup_n \max_i \sqrt[n]{(A^n)_{i,i}},
 \end{align*}
 which concludes the proof.
\end{proof}

\begin{remark}
	One can also obtain a similar formula for the joint spectral radius as in \cite{bui2024bound} using this method with the construction in Section~\ref{sec:reduction}. However, it would be more complicated to argue.
\end{remark}
\subsection{Finiteness property and linear patterns}
The growth rate can be written a bit differently as
\[
	\lambda= \sup_{\text{linear pattern $P$}} \sqrt[|P|]{\sup_n \max_i \sqrt[n]{[M(P)^n]_{i,i}}},
\]
since the linear pattern $P^n$ has the associated matrix $M(P)^n$ and satisfies $|P^n|=n|P|$.

By the formula of the spectral radius in Theorem~\ref{thm:spectral-radius}, we have
\[
	\lambda=\sup_{\text{linear pattern $P$}} \sqrt[|P|]{\rho(M(P))}.
\]

We call $\bar\lambda_P=\sqrt[|P|]{\rho(M(P))}$ the rate of the pattern $P$.
The formula for the new notation is
\begin{equation}\label{eq:in-rate-pattern}
	\lambda=\sup_{\text{linear pattern $P$}} \bar\lambda_P.
\end{equation}

The rate of a linear pattern is the original motivation for the proof of the limit $\lambda$ in the positive setting in \cite{bui2021growth}. Although it is not technically more important than Theorem~\ref{thm:limsup=sup}, its meaning is worth mentioning: Consider the sequence of the trees of $P^1, P^2, \dots$, the vectors $v^{(1)},v^{(2)},\dots$ associated with these trees are $Ms, M^2s, \dots$ for $M=M(P)$. As $s>0$, the growth $\lambda_P=\lim_{n\to\infty} \sqrt[n]{\|v^{(n)}\|}$ of the norms $\|v^{(n)}\|$ is the spectral radius of $M$. However, a lower bound on the growth rate should be $\rho(M)$ after being normalized, by taking the $|P|$-th root, as the number of leaves in $P_i$ grows by $|P|$ in each step, that is $\lambda\ge\bar\lambda_P=\sqrt[|P|]{\lambda_P}$. The proof in \cite{bui2021growth} manages to show that this is also the upper bound.

Representing the growth rate in terms of the rates of linear patterns gives some new insight. While the supremum is rarely attained in the form of Theorem~\ref{thm:limsup=sup} (as rarely as in the case of Theorem~\ref{thm:spectral-radius}), it is quite common that some linear pattern attains the growth rate in the form of \eqref{eq:in-rate-pattern}, that is $\lambda=\bar\lambda_P$ for some $P$. For example, the system in \cite[Theorem~$3$]{bui2021growth}, where $s=(1,1)$ and $x*y=(x_1y_2+x_2y_1,x_1y_2)$, takes the golden ratio as the growth rate, and the growth rate is attained by a linear pattern where the tree has two leaves with the marked leaf on the left. The readers can also check \cite{rote2019maximum} and \cite{rosenfeld2021growth} for more complicated examples.

One can pose the following natural question.
\begin{question}
	When is the growth rate $\lambda$ actually the rate of a linear pattern? 
\end{question}

We relate this question to the finiteness property of a set of matrices.
Given a pair of matrices $A,B$ and the associated bilinear system that is constructed as in Section~\ref{sec:reduction}, we have
\[
	\lambda=\sqrt[3]{\rho(\{A,B\})}.
\]

Suppose the pair $A,B$ has the finiteness property, that is there exists a sequence $M_1,\dots,M_m$ where each matrix is in $\{A,B\}$ so that $\sqrt[m]{\rho(M_1\dots M_m)}=\rho(\{A,B\})$. We can then build a pattern $P=(T,\ell)$ so that $\bar\lambda_P=\sqrt[3]{\rho(\{A,B\})}$. Indeed, if $T'$ is the tree of $3m$ leaves that is associated to $M_1\dots M_m$ (as in Theorem~\ref{thm:undecidable-estimate}), we can let $T$ be the tree of $3m+1$ leaves where one branch is $T'$ and the other branch is the marked leaf $\ell$. The readers can check that $\bar\lambda_P=\sqrt[3]{\rho(\{A,B\}})$.

On the other hand, suppose the pair $A,B$ does not have the finiteness property, e.g. the class of pairs in \cite{blondel2003elementary}, or an explicit instance in \cite{hare2011explicit}. In this case, there is no linear pattern where $\bar\lambda_P=\sqrt[3]{\rho(\{A,B\})}$, since otherwise, by considering the sequence of $P^t$ for $t=1,2,\dots$, we would have a periodic sequence of products of matrices whose norms follow the rate $\rho(\{A,B\})$ (with respect to the number of matrices).

In fact, the readers can find in \cite[Theorem~$2$]{bui2021growth} a simple example in the positive setting where no linear pattern has the same rate as the growth rate $\lambda$. The example is not related to the joint spectral radius and involves only binary entries and coefficients with $s=(1,1)$ and $x*y=(x_1y_1+x_2y_2,x_2y_2)$. 
On the other hand, the algebraic nature of the entries in the example of \cite{hare2011explicit} is quite complicated. It seems that $\prob{JSR}\le \prob{GRBS}$ suggests that some phenomenon of $\prob{GRBS}$ may be easier to construct than a similar one of $\prob{JSR}$. Nevertheless, the finiteness conjecture is still open for the case of rational (and equivalently binary) matrices, see \cite{jungers2008finiteness}. Note that if we have a reduction from $\prob{GRBS}$ to $\prob{JSR}$ that is as natural as the one in Section~\ref{sec:reduction} and keeps the resulting vectors in some form in the resulting matrices, then we can obtain a set of binary matrices without finiteness property. 

\subsection{Computability of the growth rate}
We prove the computability of the growth rate.

\begin{theorem}\label{thm:computable}
	The growth rate in the nonnegative setting is computable.
\end{theorem}
\begin{proof}
	It was known that the growth rate $\lambda$ in the nonnegative setting is upper semi-computable \cite{rosenfeld2021growth}, in the sense that there exists a sequence of upper bounds converging to $\lambda$. It remains to show that it is lower semi-computable, by showing a sequence of lower bounds converging to $\lambda$.
	As
	\[
		\lambda= \sup_{\text{linear pattern $P$}} \max_i \sqrt[|P|]{M(P)_{i,i}},
	\]
	we have
	\[
		\lambda=\sup_n \max_{\substack{\text{linear pattern $P$}\\ |P|=n}} \max_i \sqrt[n]{M(P)_{i,i}}.
	\]

    The sequence
    \[
        a_n=\max_{\substack{\text{linear pattern $P$}\\ |P|\le n}} \max_i \sqrt[|P|]{M(P)_{i,i}}
    \]
    for $n=1,2,\dots$ is indeed the desired sequence since it is increasing and converges to $\lambda$. 
\end{proof}

\subsection{Transform to make the limit exist}
Beside the transformation in Proposition~\ref{prop:insert-0-at-odd}, we also present the following transformation, as an application of Theorem~\ref{thm:limsup=sup}. While the former transformation makes the limit not exist, the latter ensures the opposite.

\begin{proposition}\label{prop:keep-rate-limit-ensured}
	For every bilinear system $(*,s)$ we can construct $(*',s')$ so that $(*',s')$ has the same growth rate as $(*,s)$ and the limit of $\sqrt[n]{g'(n)}$ exists, where $g'(n)$ is the function for $(*',s')$.
\end{proposition}
\begin{proof}
    We assume $\lambda>0$, otherwise it is trivial. (Note that $\lambda>0$ if and only if the dependency graph has a cycle.) 
    
	For the space $\mathbb R^d$ of $(*,s)$, consider $*':\mathbb R^{d+2}\times\mathbb R^{d+2}\to\mathbb R^{d+2},\ s'\in\mathbb R^{d+2}$ so that the coefficients of $*$ and the entries of $s$ are brought to the first $d$ dimensions of $(*',s')$. We let $s'_{d+1}=s'_{d+2}=\alpha$ where $0<\alpha\le\lambda$. We can take any positive lower bound of $\lambda$, e.g. by Theorem~\ref{thm:limsup=sup}. (In fact, the value of $s'_{d+1}$ does not matter.)
	The operator $*'$ is defined so that
	\[
		(x*'y)_{d+1}=\sum_{i=1}^d x_i y_{d+2}
	\]
	and
	\[
		(x*'y)_{d+2}=x_{d+2} y_{d+2}.
	\]

	The $(d+2)$-th entry of any vector obtained from combining $n$ instances of $s'$ is $\alpha^n$. It follows that for any index $i$ and any $\delta\ge 1$, the largest possible $(d+1)$-th entry $g'_{d+1}(n)$ over all combinations of $n$ instances of $s'$ satisfies
	\[
		g'_{d+1}(n+\delta)\ge\alpha^\delta g_i(n)
	\]
	by considering the composition tree where the left branch is associated with $g_i(n)$ and the right branch is any tree of $\delta$ leaves.
 This means that for a bounded $\delta$, we have
	\begin{equation}\label{eq:g-prime-greater-g-bounded-distance}
		g'(n+\delta)\ge g'_{d+1}(n+\delta)\ge\max_i\alpha^\delta g_i(n)=\alpha^\delta g(n).
	\end{equation}

 On the other hand,
 \[
    g'_{d+1}(n)\le d \max_{1\le i\le d} \max_{1\le \delta\le n-1} \alpha^\delta g_i(n-\delta),
 \]
 which implies
 \[
    \limsup_{n\to\infty} \sqrt[n]{g'_{d+1}(n)} \le \max\{\alpha, \limsup_{n\to\infty} \sqrt[n]{\max_i g_i(n)\}} = \max\{\alpha, \lambda\}=\lambda.
 \]
It follows that
	\[
		\limsup_{n\to\infty} \sqrt[n]{g'(n)}\le\lambda.
	\]

	For any linear pattern $P$ with the associated matrix $M$ and any index $i$, we prove that
	\begin{equation}\label{eq:lower-bound}
		\liminf_{n\to\infty} \sqrt[n]{g'(n)}\ge \sqrt[|P|]{M_{i,i}}.
	\end{equation}

	Indeed, we pick a fixed $n_0$ so that $g_i(n_0)>0$ with the associated tree $T_0$. For any $n$ large enough, we write $n=q|P|+ n_0 + r$ so that $1\le r\le |P|$. Consider the pattern $P^q$ with the associated matrix $M^q$. Since $(M^q)_{i,i}\ge (M_{i,i})^q$, the $i$-th entry of the vector associated to $P^q\oplus T_0$ is at least a constant times $(M_{i,i})^q$. As $r$ is bounded and $P^q\oplus T_0$ has $q|P|+n_0$ leaves, it follows from \eqref{eq:g-prime-greater-g-bounded-distance} that
	\[
		g'(n)\ge \alpha^r g(n-r)\ge\const g_i(q|P|+n_0)\ge\const(M_{i,i})^q.
	\]
	As $n-q|P|$ is bounded, we have proved \eqref{eq:lower-bound}. It follows that
	\[
		\liminf_{n\to\infty} \sqrt[n]{g'(n)}\ge\sup_{\text{linear pattern $P$}}\max_i \sqrt[|P|]{M(P)_{i,i}} = \lambda,
	\]
	where the equality is due to Theorem~\ref{thm:limsup=sup}.
 
	In total, we have the limit
	\[
		\lim_{n\to\infty} \sqrt[n]{g'(n)}=\liminf_{n\to\infty} \sqrt[n]{g'(n)}=\limsup_{n\to\infty} \sqrt[n]{g'(n)}=\lambda. \qedhere
	\]
\end{proof}

\paragraph{Another (but only conditional) approach to proving the undecidability of limit checking.}
Assuming Conjecture~\ref{con:rho>=1} holds, that is checking $\rho\ge 1$ and checking $\lambda\ge 1$ are undecidable, we give another approach to the undecidability of the problem of checking if the limit of $\sqrt[n]{g(n)}$ exists, as an application of Proposition~\ref{prop:keep-rate-limit-ensured}.

Given a system $(*,s)$, let the system $(*',s')$ obtained from Proposition~\ref{prop:keep-rate-limit-ensured} be in the space $\mathbb R^{d'}$. Consider $*'':\mathbb R^{d'+2}\times\mathbb R^{d'+2}\to\mathbb R^{d'+2}$ and $s''\in\mathbb R^{d'+2}$ where the first $d'$ dimensions are brought from $(*',s')$. We let $s''_{d'+1}=1$, $s''_{d'+2}=0$, and
\[
	(x*''y)_{d'+1}=x_{d'+2} y_{d'+2},\qquad (x*''y)_{d'+2}=x_{d'+1} y_{d'+1}.
\]

We can see that the last $2$ dimensions are independent of the remaining dimensions, and $\max\{g''_{d'+1}(n), g''_{d'+2}(n)\}$ is $0$ if $n$ is divisible by $3$ and it is $1$ otherwise, where $g''$ is the function for $(*'',s'')$. It follows that 
\begin{align*}
    \limsup_{n\to\infty} \sqrt[n]{g''(n)} &= \max\left\{\limsup_{n\to\infty} \sqrt[n]{\max\{g''_{d'+1}(n), g''_{d'+2}(n)\}},\limsup_{n\to\infty} \sqrt[n]{g'(n)}\right\} \\
    &= \max\{1,\lambda\}.
\end{align*}
Meanwhile,
\[
    \liminf_{n\to\infty} \sqrt[n]{g''(n)} \ge \liminf_{n\to\infty} \sqrt[n]{g'(n)} = \lim_{n\to\infty} \sqrt[n]{g'(n)} = \lambda,
\]
and since $g''(3m)=g'(3m)$ for any $m$, we have
\[
    \liminf_{n\to\infty} \sqrt[n]{g''(n)} \le \liminf_{m\to\infty} \sqrt[3m]{g''(3m)} = \liminf_{m\to\infty} \sqrt[3m]{g'(3m)} = \lim_{n\to\infty} \sqrt[n]{g'(n)} = \lambda.
\]
In total,
$\liminf_{n\to\infty} \sqrt[n]{g''(n)} = \lambda$. It follows that
we have reduced the problem of checking $\lambda\ge 1$ to the problem of checking if the limit of $\sqrt[n]{g''(n)}$ exists.
Therefore, the latter problem is undecidable, under the assumption on the undecidability of $\lambda\ge 1$.

\section{Multiple operators and multiple starting vectors}
\label{sec:multiple-operators-vectors}
Rosenfeld \cite{rosenfeld2022undecidable} made a remark that the problem of the bilinear system does not become harder when we allow multiple operators and multiple starting vectors. We confirm this remark.

The construction in Section~\ref{sec:reduction} is well suited for reducing the problem for $(*,\{s,s'\})$ to the original problem. By the problem for $(*,\{s,s'\})$ we mean the problem where we can choose either $s$ or $s'$ in the place of each $s$ instead of fixing the vector $s$. The two vectors $s,s'$ play the roles of $A,B$ in the construction. We rewrite it formally without repeating the verification.

For a bilinear map $*:\mathbb R^d\times\mathbb R^d\to\mathbb R^d$ and two vectors $s,s'\in\mathbb R^d$, consider 
the system $(\bullet,u)$ with the $(d+2)$-dimensional vector $u=(\mathbf{0},1,0)$ where $\mathbf{0}$ is the $d$-dimensional zero vector and $\bullet:\mathbb R^{d+2}\times\mathbb R^{d+2}\to\mathbb R^{d+2}$ presented by
\begin{equation*}
    \begin{pmatrix}
        w\\i\\j
    \end{pmatrix}
    \bullet
    \begin{pmatrix}
        w'\\i'\\j'
    \end{pmatrix}
    =
    \begin{pmatrix}
        w*w' + ij's +ji's'\\
        0\\
        ii'
    \end{pmatrix}
    .
\end{equation*}

By the same analysis as in Theorem~\ref{thm:undecidable-estimate}, the growth rate of $(\bullet,u)$ is the cube root of the growth rate of $(*,\{s,s'\})$.

Using the idea of the previous construction, we can reduce the problem for $(\{*,*'\},s)$ to the original problem. By the problem for $(\{*,*'\},s)$ we mean the problem where we can choose either $*$ or $*'$ in the place of each instance of $*$ instead of fixing $*$.

For two bilinear maps $*,*':\mathbb R^d\times\mathbb R^d\to\mathbb R^d$ and a vector $s\in\mathbb R^d$, consider the system $(\bullet, u)$ with the $(3d+2)$-dimensional vector $u=(s,s,\mathbf{0},1,0)$ where $\mathbf{0}$ is the $d$-dimensional zero vector and $\bullet: \mathbb R^{3d+2}\times\mathbb R^{3d+2}\to\mathbb R^{3d+2}$ presented by
\begin{equation*}
    \begin{pmatrix}
        x\\y\\w\\i\\j
    \end{pmatrix}
    \bullet
    \begin{pmatrix}
        x'\\y'\\w'\\i'\\j'
    \end{pmatrix}
    =
    \begin{pmatrix}
        w*w'\\
        w*'w'\\
        jx'+yj'\\
        0\\
        ii'
    \end{pmatrix}
    .
\end{equation*}

We sketch the approach: For any vector $v$ obtained from combining $n$ instances of $u$ using $\bullet$, if $v_{[2d+1,3d]}\ne 0$ then $n=5k+3$ for some $k$. Also, if $v_{[1,d]}$ or $v_{[d+1,2d]}$ is not a zero vector, then $n=5k+1$ for some $k$. The growth rate of $(\bullet, u)$ is the fifth root of the growth rate of $(\{*,*'\},s)$. The verification is similar to that in Theorem~\ref{thm:undecidable-estimate} and we leave it to the readers.

A construction for a higher number of starting vectors or a higher number of bilinear operators, or both, can be established similarly by introducing more dimensions. We leave it to the readers as an exercise since the details would be tedious.

In conclusion, introducing more vectors and more operators does not make the problem any harder.

\section{Undecidability of checking $\lambda\le 1$ in the positive setting}
\label{sec:undecidable-positive-setting}
As we can reduce the problem of checking $\rho\le 1$ for the joint spectral radius $\rho$ to the problem of checking $\lambda\le 1$ for the growth of bilinear maps in the nonnegative setting, one may wonder if there is a similar reduction for the positive setting, where all the entries of $s$ have to be positive. In this section, we give such a reduction, which implies the undecidability of checking $\lambda\le 1$ in the positive setting, by the undecidability of the following problem.
\begin{theorem}[Rote 2024]\label{thm:undecidable-jsr-positive}
	It is undecidable to check $\rho(\{A,B\})\le 1$ for the joint spectral radius $\rho$ of a pair of \emph{positive} matrices $A,B$.
\end{theorem}
The undecidability follows from \cite[Theorem~$4$]{rote2024probabilistic}, which is a variant of the problem Probabilistic Finite Automaton Emptiness. Discussions on the reduction can be found in \cite[Section $3.2$]{rote2024probabilistic}.

The reduction from the problem in Theorem~\ref{thm:undecidable-jsr-positive} to the problem of bilinear systems is almost the same as the one in Section~\ref{sec:reduction} but with some ideas of the reduction in Section~\ref{sec:mortality} and a more complicated argument.

We reuse the convention of embedding a matrix into a vector in Section~\ref{sec:reduction}. For a pair of $d\times d$ \emph{positive} matrices $A,B$, we consider the system $(*,s)$ with the $(d^2+2)$-dimensional vector $s=(\mathbf E, 1, \epsilon)$ where $\mathbf E$ denotes\footnote{$\mathbf E$ here is the capital version of $\epsilon$, for the mnemonic purpose.} the $d\times d$ matrix with all entries set to $\epsilon$ and $\epsilon>0$ is small enough. The operator $*:\mathbb R^{d^2+2}\times \mathbb R^{d^2+2}\to \mathbb R^{d^2+2}$ is presented by
\begin{equation}\label{eq:presentation-of-*}
    \begin{pmatrix}
        C\\i\\j
    \end{pmatrix}
    *
    \begin{pmatrix}
        C'\\i'\\j'
    \end{pmatrix}
    =
    \begin{pmatrix}
        CC' + ji'X + ij'Y \\
        0\\
        ii'
    \end{pmatrix}
    ,
\end{equation}
where $X,Y,\epsilon$ satisfy some requirements that are given in \eqref{eq:requirements} below.

Let us analyze some beginning combinations of $s$:
\begin{align*}
    s&=(\mathbf E, 1, \epsilon),\\
    s*s&=(\mathbf E^2 + \epsilon X + \epsilon Y,0,1),\\
    (s*s)*s&=((\mathbf E^2 + \epsilon X + \epsilon Y)\mathbf E + X,0,0),\\
    s*(s*s)&=(\mathbf E(\mathbf E^2 + \epsilon X + \epsilon Y) + Y, 0,0)\\
    (s*s)*(s*s)&=((\mathbf E^2 + \epsilon X + \epsilon Y)^2,0,0).
\end{align*}

We need $X,Y$ be so that
\begin{gather}\label{eq:requirements}
\begin{split}
	(s*s)*s=(A,0,0),\\
	s*(s*s)=(B,0,0),\\
	X\ge 0,\qquad Y\ge 0.
\end{split}
\end{gather}
The requirements $X\ge 0, Y\ge 0$ are for the coefficients of $*$ to be nonnegative.\footnote{Note that $X,Y$ can be chosen to be positive. However, the positivity is not useful in case one wishes to make all the coeffcients of $*$ to be positive (there are already some zero coefficients in the representation of $*$ in \eqref{eq:presentation-of-*}).}
\begin{proposition}
    Such $X,Y$ always exist for any $\epsilon$ small enough.
\end{proposition}
\begin{proof}
The first two conditions in \eqref{eq:requirements} are 
\begin{align*}
    (\mathbf E^2 + \epsilon X + \epsilon Y)\mathbf E + X &= A,\\
    \mathbf E(\mathbf E^2 + \epsilon X + \epsilon Y) + Y &= B.
\end{align*}
They are actually $2d^2$ linear equations of a system, where the entries of $X$ and $Y$ are $2d^2$ unknowns.
Since $\epsilon=0$ has obviously a unique solution $X=A,Y=B$, the solution must
be also unique for any small enough $\epsilon$ (by the continuity of the system
determinant).
The solution of $X,Y$ also satisfies $X\ge 0$ and $Y\ge 0$ when $\epsilon$ is small enough.
\end{proof}

Let $\Gamma(v)$ be the matrix form of the first $d^2$ dimensions of a vector $v$. 
Denote $M_1=\Gamma(s)=\mathbf E$ and $M_2=\Gamma(s*s)=\mathbf E^2 + \epsilon X + \epsilon Y$, we have both $M_1<\mathbf E'$ and $M_2 < \mathbf E'$ where $\mathbf E'$ is the matrix of all entries $\epsilon'$ that depends on $\epsilon$. The value $\epsilon'$ can be made arbitrarily small by reducing $\epsilon$.

We make the following observation, whose verification is simple and left to the readers.
\begin{proposition}
	The matrix form $\Gamma(v)$ for any vector $v$ obtained by combining $n$ instances of $s$ is the product of some matrices from $\{A,B,M_1,M_2\}$. In particular, if $m_A,m_B,m_1,m_2$ are respectively the numbers of instances of $A,B,M_1,M_2$, then $m_1 + 2m_2 + 3(m_A+m_B) = n$. On the other hand, for any product of $m$ matrices from $\{A,B\}$, we have a combination for $n=3m$ so that $\Gamma(v)$ is the product.
\end{proposition}

Since $\epsilon'$ can be made arbitrarily small, the numbers $m_1,m_2$ should be made minimal to maximize the norm of $v$. It follows that $\lambda=\sqrt[3]{\rho(\{A,B\})}$ like in Theorem~\ref{thm:undecidable-estimate}. Therefore, the problem of checking $\lambda\le 1$ is undecidable in the positive setting.
\begin{remark}
    In contrast to the situation in Section~\ref{sec:reduction}, the limit of $\sqrt[n]{g(n)}$ here exists, because this is always the case for a positive setting by Corollary~\ref{cor:lim-exists}.
\end{remark}

\subsection*{Acknowledgments}
The author would like to thank G\"unter Rote for interesting and helpful discussions, and Matthieu Rosenfeld for suggesting various meaningful questions. Part of the work was included in the PhD thesis of the author done at Freie Universit\"at Berlin under the supervision of G\"unter Rote. The author also appreciates the two anonymous referees for their detailed and helpful comments on the manuscript.

\bibliographystyle{unsrt}
\bibliography{gbm3}

\end{document}